\documentclass[12pt]{amsart}
\usepackage{amssymb}
\usepackage{enumerate}

\usepackage{graphicx}
\usepackage[usenames,dvipsnames,svgnames,table]{xcolor}
\usepackage[a4paper]{geometry}
\usepackage{psfrag,xmpmulti,amscd,color,pstricks, import}

\newcommand{\eps}{\varepsilon}

\newcommand{\K}{{\mathcal K}}

\newcommand{\N}{{\mathbb N}}
\newcommand{\C}{{\mathbb C}}
\newcommand{\D}{{\mathbb D}}
\newcommand{\Ha}{{\mathbb H}}
\newcommand{\Z}{{\mathbb Z}}
\newcommand{\R}{{\mathbb R}}

\newcommand{\tef}{transcendental entire function}

\newcommand\qfor{\;\;\text{for }}

\theoremstyle{plain}
\newtheorem{theorem}{Theorem}[section]
\newtheorem{corollary}[theorem]{Corollary}

\newtheorem*{theorem*}{Theorem}
\newtheorem*{proposition*}{Proposition}

\newtheorem{lemma}[theorem]{Lemma}
\theoremstyle{definition}
\newtheorem{definition}[theorem]{Definition}
\theoremstyle{remark}
\newtheorem*{remark*}{Remark}
\newtheorem*{remarks*}{Remarks}

\theoremstyle{problem}

\theoremstyle{example}
\newtheorem{example}[theorem]{Example}
\newtheorem*{example*}{Example}
\theoremstyle{question}

\theoremstyle{questions}
\newtheorem*{questions*}{Questions}

{\begin{list}{}%
         {\setlength{\leftmargin}{#1}}%
         \item[]%
}
{\end{list}}

\begin{document}


\title[On subharmonic and entire functions]{On subharmonic and entire functions of small order: after Kjellberg}

\author{P. J. Rippon}
\address{School of Mathematics and Statistics \\
The Open University \\
   Walton Hall\\
   Milton Keynes MK7 6AA\\
   UK}
\email{phil.rippon@open.ac.uk}

\author{G. M. Stallard}
\address{School of Mathematics and Statistics \\
The Open University \\
   Walton Hall\\
   Milton Keynes MK7 6AA\\
   UK}
\email{gwyneth.stallard@open.ac.uk}

\thanks{Both authors were supported by the EPSRC grant EP/R010560/1.}



\begin{abstract}
We give a general method for constructing examples of transcendental entire functions of given small order, which allows precise control over the size and shape of the set where the minimum  modulus of the function is relatively large. Our method involves developing a new technique to give an upper bound for the growth of a positive harmonic function defined in a certain type of multiply connected domain, giving a sharp estimate for the growth in many cases.
\end{abstract}
\maketitle

\section{Introduction}\label{intro}
\setcounter{equation}{0}

This paper concerns transcendental entire functions of small order. Such functions have been studied extensively in classical complex analysis, ever since Wiman \cite{aW05} observed that such functions have properties which, in some ways, resemble those of polynomials. Subsequently, powerful results such as the version of the cos\,$\pi \rho$ theorem due to Barry \cite{Ba63} showed that, for such functions, the minimum modulus of the function on many circles centred at the origin is comparable in size to the maximum modulus of the function.

More recently, these properties have led to functions of small order playing a key role in two major conjectures in complex dynamics: Baker's conjecture explicitly concerns such functions (see \cite{NRS18} for recent progress on this conjecture) and they were also shown (in \cite{RS09a} and \cite{RS12}) to have an unexpected link with Eremenko's conjecture \cite{E}, one of the main drivers of research in transcendental dynamics, arising from the fact that for functions of small order the escaping set of the function is often connected; see \cite{NRS19} and \cite{RS12}.

In this paper we give a very general method for constructing examples of functions of given small order, including order~0, which allows precise control over the size and shape of the set where the minimum  modulus of the function is relatively large. The original motivation for this work was related to further progress on Eremenko's conjecture, which we report on in forthcoming work \cite{NRS20}, but our new results here have wider applications.

For a transcendental entire function~$f$ the {\it maximum modulus} and {\it minimum modulus} of $f$ are denoted by
\[
M(r) = M(r,f)=\max_{|z|=r} |f(z)| \quad\text{and}\quad m(r) = m(r,f)=\min_{|z|=r} |f(z)|,
\]
respectively. The {\it order} of~$f$ is
\[
\limsup_{r \to \infty} \frac{\log \log M(r)}{\log r}
\]
and we say that a function has {\it small order} if it has order less than~1/2. Functions of order 1/2, {\it minimal type}, that is, functions of order 1/2 with $\log M(r) = o(r^{1/2})$ as $r \to \infty$, are sometimes included in this class, since they also have many circles on which the minimum modulus is relatively large (though far fewer such circles than for functions of order less than~$1/2$).

Here we give a considerable generalisation of Kjellberg's construction \cite{bK48} of {\tef}s with orders in the range $(0,1/2)$, which he used to show that various theorems about the minimum modulus of {\tef}s of order less that 1/2 are in a certain sense best possible.

Kjellberg's construction approximates continuous subharmonic functions with specified properties by functions of the form $\log |f|$, where~$f$ is a {\tef}, with similar properties. His construction is in two stages:
\begin{itemize}
\item[1.]
a continuous subharmonic function~$u$ with the required properties is specified by using a positive harmonic function defined on a domain whose complement is a union of radial slits, on which~$u$ vanishes;
\item[2.]
the Riesz measure of~$u$ is discretised to produce an entire function~$f$ such that $\log |f|$ is close to~$u$ away from the zeros of~$f$.
\end{itemize}
In Kjellberg's original construction, the radial slits were chosen to lie on a ray from the origin in such a way that the union of the slits is invariant under a scaling of the plane. This invariance was a key property to allow various parts of his reasoning to succeed.

Here we give a method which allows the slits to be chosen much more flexibly. This enables us to construct examples of entire functions with prescribed order~$\rho$, for each $\rho$ in the {\it closed} interval $[0,1/2]$, which also have bounded minimum modulus on as large a set as possible given the order. To achieve the necessary control, we introduce a new technique for estimating the growth of certain positive harmonic functions from above; this technique may well have applications beyond our current purpose.

We introduce a class of subharmonic functions named $\K$, after Kjellberg.

\begin{definition}
A subharmonic function~$u$ is in the class ${\mathcal K}$ if $u$ is continuous in $\C$ and positive harmonic in $D=\C\setminus E$, where $E\subset (-\infty,0]$ is a closed set on which $u$ vanishes. We assume that each point of~$E$ is regular for the Dirichlet problem in~$D$.

We denote by $u_{\alpha,\beta}$ the function in $\K$ corresponding to the set
\[
E=\{0\}\cup\bigcup_{n\in \Z}[-\alpha\beta^n,-\beta^n],
\]
where $1<\alpha<\beta$, and $u_{\alpha,\beta}(1)=1$.
\end{definition}

{\it Remarks}\; 1. For each closed subset~$E$ of the negative real axis, there is exactly one corresponding function $u\in \K$, up to positive scalar multiples, by a result of Benedicks \cite[Theorem~4]{Ben}, and such a~$u$ is unbounded and symmetric with respect to the real axis.

2. The functions $u_{\alpha,\beta}$ correspond to the functions considered by Kjellberg in his original construction \cite[Chapter 3]{bK48}. His set~$E$ was of the form $E=\{0\}\cup\bigcup_{n\in \Z}[\beta^n,\alpha\beta^n]$, where $1<\alpha<\beta$, whereas for us it is on balance more convenient to have the set~$E$ contained in the negative real axis.

To state our results, we recall that for a set $S\subset \R^+$ and $r>1$, the {\it upper logarithmic density} of~$S$ is
\[
\overline{\Lambda}(S) = \limsup_{r\to \infty}\frac{1}{\log r}\int_{S\cap(1,r)}\frac{dt}{t},
\]
and the {\it lower logarithmic density} of~$S$ is
\[
\underline{\Lambda}(S) = \liminf_{r\to \infty}\frac{1}{\log r}\int_{S\cap(1,r)}\frac{dt}{t}.
\]
When $\overline{\Lambda}(S)=\underline{\Lambda}(S)$ we speak of the {\it logarithmic density} of $S$, denoted by $\Lambda(S)$.

We also define, for a continuous subharmonic function~$u$ in $\C$,
\[
A(r)=A(r,u)=\min_{|z|=r} u(z)\quad\text{and}\quad B(r)=B(r,u)=\max_{|z|=r} u(z),
\]
and the {\it order} and {\it lower order} of $u$,
\[
\rho(u)=\limsup_{r\to\infty}\frac{\log B(r)}{\log r}\quad \text{and}\quad \lambda(u)=\liminf_{r\to\infty}\frac{\log B(r)}{\log r},
\]
respectively.

First, we give some basic properties of {\it all} functions in the class $\K$.
\begin{theorem}\label{basic-props}
Let $u\in \K$, with $E$ the corresponding closed subset of the negative real axis. Then $u$ has the following properties.
\begin{itemize}
\item[(a)]
Monotonicity properties:
for all $r>0$,
\[
u(re^{i\theta})\;\;\text{is decreasing as a function of } \theta, \text{ for } 0\le \theta \le \pi,
\]
so, in particular, $B(r,u)=u(r)$ and $A(r,u)=u(-r)$ for all $r>0$. Also,
\[
\frac{u(r)}{r^{1/2}}=\frac{B(r)}{r^{1/2}}\;\;\text{is decreasing for } r>0,
\]
so, in particular, $\rho(u)\le 1/2$.
\item[(b)] Bounds for order and lower order:
\[
\rho(u)\ge \tfrac12 \overline{\Lambda}(E^*)\quad {\text and} \quad \lambda(u)\ge \tfrac12 \underline{\Lambda}(E^*),
\]
where $E^*=\{|x|:x\in E\}$.
\end{itemize}
\end{theorem}
It is natural to ask whether equality holds for the lower bounds in part~(b). For the functions $u_{\alpha,\beta}$, $1<\alpha<\beta$, Kjellberg proved that
\[
\frac{\log \alpha}{2\log \beta}\le \rho(u_{\alpha,\beta})=\lambda(u_{\alpha,\beta})<\frac12.
\]
Here the lower bound is the one given by Theorem~\ref{basic-props}, part~(b) and Kjellberg showed that equality does not hold in general; see the discussion below, after Theorem~\ref{exactorder}. He obtained the strict upper bound of~$1/2$ by using the invariant nature of his set~$E$; see Section~\ref{min-type-result}, Remark~2 for a proof of this upper bound for the functions $u_{\alpha,\beta}$.

In our main result, which follows, we show that for many functions $u\in\K$, corresponding to sets~$E$ that are unions of closed intervals, the order and lower order of~$u$ can be expressed explicitly in terms of the geometric properties of the set~$E$. We do this by developing a new technique to give upper bounds for the growth of positive harmonic functions defined in multiply connected domains, using a result about the relationship between the Harnack metric and the hyperbolic metric in such a domain \cite{dH87,hBwS}, together with the Beardon--Pommerenke estimate for the density of the hyperbolic metric \cite{aBcP}.
\begin{theorem}\label{exactorder}
Suppose that $u\in\K$ and
\[
E=\bigcup_{n\ge 0}[-b_n,-a_n],
\]
where $0\le a_0<b_0<a_1<b_1< \cdots,$ and $a_n \to \infty$ as $n\to\infty$. If
\begin{equation}\label{an-cond}
a_n^{1/n} \to \infty\;\text{ as }n\to \infty,
\end{equation}
then
\begin{equation}\label{order-exact}
\rho(u)=\tfrac12\overline{\Lambda}(E^*)\quad {\text and} \quad \lambda(u)= \tfrac12 \underline{\Lambda}(E^*).
\end{equation}
\end{theorem}
We note that some condition such as \eqref{an-cond} is essential here. Indeed, as mentioned earlier, Kjellberg showed that the functions $u_{\alpha,\beta}\in\K$ do not in general satisfy the identities in \eqref{order-exact}. He did this by proving that if  the parameters~$\alpha$ and~$\beta$ tend to 1 while $\frac12\log \alpha/\log \beta$ remains constant (so the intervals in the set~$E$ and their complementary intervals become ever slimmer while the logarithmic density of~$E^*$ remains fixed), then the order of $u_{\alpha,\beta}$ must tend to~1/2.

By using a precise harmonic measure estimate in multiply connected domains due to Sodin \cite{Sodin}, we can construct an even more extreme example to demonstrate this phenomenon.

\begin{example}\label{extreme-ex}
There exists a function $u\in \K$ with the set $E$ of the form
\[
E=\bigcup_{n\ge 0}[-b_n,-a_n],
\]
where $0\le a_0<b_0<a_1<b_1< \cdots,$ and $a_n \to \infty$ as $n\to\infty$, such that $\overline{\Lambda}(E^*)=0$,
\[
\rho(u)= \lambda(u)=1/2\quad\text{and moreover}\quad \lim_{r\to\infty}\frac{u(r)}{r^{1/2}}>0.
\]
\end{example}

In forthcoming work \cite{NRS20} we will use the results in this paper to construct entire functions of order 1/2, minimal type, with dynamically interesting properties related to their minimum modulus. Our next theorem is useful in any situation where we need to construct examples of order 1/2, minimal type.
\begin{theorem}\label{min-type}
Let $u\in \K$, with $E$ the corresponding closed subset of the negative real axis. If
\[E^c\supset\bigcup_{n\ge 0}(-d_n,-c_n),\]
where $0\le c_0<d_0<c_1<d_1< \cdots,$ and $\limsup_{n\to\infty} d_{n}/c_n>1$, then
\[
\frac{u(r)}{r^{1/2}}\to 0\;\;\text{as }r\to\infty.
\]
\end{theorem}

Our final theorem is the result needed in the second stage of Kjellberg's process, which shows how to approximate a function $u\in\K$ by $\log |f|$, where $f$ is entire. This generalises the result given by Kjellberg in \cite[Chapter~4]{bK48} for his particular type of set $E$. Due to the much greater generality of the set~$E$ considered here, the proof is significantly more delicate, so we include full details.
\begin{theorem}\label{discretise}
Suppose that $u\in\K$ and
\[
E=\bigcup_{n\ge 0}[-b_n,-a_n],
\]
where $0\le a_0<b_0<a_1<b_1< \cdots,$ and $a_n \to \infty$ as $n\to\infty$. Put
\[
D_1=\C\setminus \{z: \text{{\rm dist}}(z,E)\le 1\}.
\]
Then there exists a {\tef}~$f$ with only negative zeros, all lying in the set~$E$, such that
\begin{equation}\label{R-est}
\log|f(z)|-u(z)=O\left(\log|z|\right) \;\;\text{as }z\to \infty, \qfor z\in D_1.
\end{equation}
Moreover, if we also have
\begin{equation}\label{d-cond}
b_n/a_n \ge d >1,\qfor n\ge 0,
\end{equation}
then there exists $R=R(u)>0$ such that
\begin{equation}\label{f-upper-est}
\log|f(z)|\le u(z)+4\log |z|, \qfor |z|\ge R.
\end{equation}
\end{theorem}
Since the work of Kjellberg there have been many results on the approximation of subharmonic functions by logarithms of moduli of entire functions; see, for example, \cite[Chapter~10]{wH89} for subharmonic functions whose Riesz measure lies on a finite number of unbounded curves,  \cite{Y} and \cite{dD01} for functions subharmonic in $\mathbb C$ of finite order, and \cite{LM01} and \cite{FR14} for functions subharmonic in $\mathbb C$ of infinite order. These works all give estimates of the form \eqref{R-est} either outside certain exceptional sets or on average in a certain sense, with various error bounds, but we are not aware of earlier results that provide the type of control of the entire function~$f$ given in \eqref{R-est} and \eqref{f-upper-est} simultaneously.

Indeed, the two estimates in Theorem~\ref{discretise} enable us to use any subharmonic function $u \in\K$ to obtain an entire function $f$ with the same order, lower order and type class as~$u$, and also with the property that~$\log|f|$ is uniformly bounded by~$u$, provided \eqref{d-cond} holds.

Finally, we recall that one aim of Kjellberg's work in \cite{bK48} was to show that a certain density estimate appearing in an early version of the $\cos \pi \rho$ theorem is best possible. Recall the strong version of the $\cos \pi \rho$ theorem due to Barry (see \cite{Ba63} or \cite[Theorem 6.13]{wH89}), which states that if~$u$ is a non-constant subharmonic function of order $\rho\in [0,1)$ and $\rho<\alpha<1$, then
\begin{equation}\label{Barry-order}
\underline{\Lambda}(\{r: A(r,u)>\cos (\pi\alpha) B(r,u)\}) \ge 1-\rho/\alpha.
\end{equation}
Kjellberg's examples in \cite{bK48} show that if $0\le \rho <\alpha =1/2$, then the logarithmic density of the set in \eqref{Barry-order} can be arbitrarily close to the quantity $1-\rho/\alpha= 1-2\rho$, thus demonstrating that the inequality in \eqref{Barry-order} for the lower logarithmic density is best possible in the case $\alpha=1/2$.

Theorem~\ref{exactorder} shows that for a given $\rho<1/2$ and $\alpha=1/2$ the value $1-\rho/\alpha=1-2\rho$ for the logarithmic density can in fact be {\it attained} by a subharmonic function~$u$ of order~$\rho$. Moreover, Theorem~\ref{discretise} allows us to use such a subharmonic function to construct an entire function with the same properties.

\begin{corollary}\label{entire-density}
For each $\rho$, $0\le \rho<1/2$, there is a {\tef} $f$ of order $\rho$ such that
\[
\underline{\Lambda}(\{r: A(r,\log|f|)>0)\}) = \overline{\Lambda}(\{r: A(r,\log|f|)>0)\}) = 1-2\rho.
\]
\end{corollary}

The structure of the paper is as follows. In Section~\ref{prelim} we state a number of key results that are needed in our proofs. Then we prove Theorem~\ref{basic-props} in Section~\ref{lower-estimate-proof}, Theorem~\ref{exactorder} in Section~\ref{upper-estimate-proof}, Example~\ref{extreme-ex} in Section~\ref{example-proof}, Theorem~\ref{min-type} in Section~\ref{min-type-result}, and finally Theorem~\ref{discretise} and Corollary~\ref{entire-density} in Section~\ref{Kjellbergproof}.

{\it Acknowledgements}\; The authors thanks Dan Nicks and Ian Short for helpful comments.

\section{Preliminary results}\label{prelim}
\setcounter{equation}{0}

Our results depend on two entirely different techniques that will enable us to estimate the growth of functions in class $\K$ from below and from above. The first is a lemma of Beurling~\cite[page~95]{aB33}, which gives estimates of growth from below and was a key tool in~\cite{bK48}, and in many other papers. Recall that for any subharmonic function~$u$ we write
\[
B(r,u)=\max_{|z|=r}u(z),\;\text{where } r>0.
\]

\begin{lemma}\label{Beur}
If~$u$ is subharmonic in~$\C$, $0<r_1<r_2$, and
\[
E(r_1,r_2)=\{r\in [r_1,r_2]:\inf_{|z|=r}u(z)\le 0\},
\]
then
\begin{equation}\label{Beur-est}
B(r_2,u)> \frac12 \exp\left(\frac12\int_{E(r_1,r_2)} \frac{dt}{t}\right)B(r_1,u).
\end{equation}
In particular, if $E(r_1,r_2)=[r_1,r_2]$, then
\[
B(r_2,u) > \frac12 \left(\frac{r_2}{r_1}\right)^{1/2} B(r_1,u).
\]
\end{lemma}
The second technique, which will give us estimates from above, needs more preparation; we are not aware of this technique being used previously to estimate the growth of positive harmonic functions from above in multiply connected domains.

The {\it Harnack metric} is defined in a domain $G$ by
\begin{equation}\label{Har-metric}
d_G(z_1,z_2)=\sup\{|\log (u(z_2)/u(z_1))|: u \text{ is positive and harmonic in } G\},
\end{equation}
where $z_1,z_2\in G$. This concept was introduced by Bear \cite{hB65} and named the Harnack metric by K\"onig \cite{hK69}. The Harnack metric in~$G$ has a close relationship with the hyperbolic metric in $G$, which we denote by $\rho_G$. Indeed, if $G$ is simply connected then these two metrics are identical, provided that $\rho_G$ is normalised so that the hyperbolic density in the unit disc $\D$ is $d\rho_{\D}(z)=2/(1-|z|^2$), or equivalently $d\rho_{\Ha}(z)=1/\Re (z)$ in the right-half plane $\Ha$.

In \cite{dH87}, and also \cite{hBwS}, the relationship between the two metrics when $G$ is multiply connected was investigated and, amongst other results, the following was obtained; see \cite[Theorem~6]{dH87} and \cite[Theorem~1.1]{hBwS}.

\begin{lemma}\label{BearSmith}
Let $G$ be a domain with Harnack metric $d_G$. Then
\[
d_G(z_1,z_2)\le \rho_G(z_1,z_2),\qfor z_1,z_2\in G.
\]
\end{lemma}
It follows that we can make good estimates for the growth of a positive harmonic function in a domain whenever we can obtain good estimates for the hyperbolic metric in that domain. To do this we shall use the following result of Beardon and Pommerenke; see \cite[Theorem~1]{aBcP}.
\begin{lemma}\label{BeardPomm}
Let~$G$ be a domain in $\C$ that omits at least two finite points. Then the hyperbolic density in $G$ satisfies
\[
d\rho_G(z) \le \frac{\pi/2}{{\rm dist}\,(z,\partial G)\beta_G(z)},
\]
where
\[
\beta_G(z)=\inf\{\left|\log|z-a|/|b-a|\right|:a,b\in \partial G,|z-a|={\rm dist}\,(z,\partial G)\}.
\]
\end{lemma}
Note that the constant $\pi/2$ in Lemma~\ref{BeardPomm} appears as $\pi/4$ in \cite[Theorem~1]{aBcP} due to the different normalisation of the hyperbolic density in \cite{aBcP}.

In our situation, where we are estimating the growth of functions in class~$\K$, the domains where the positive harmonic functions are defined always have their boundaries lying entirely in the negative real axis, so we can make use of the following special case of a deep result of Weitsman; see \cite[Theorem~9.16]{wH89}.
\begin{lemma}\label{Weit}
Let $G$ be a domain in $\C$ whose complement contains at least two points, and suppose that $G$ has the symmetry property that for each $r>0$ the set $\{z\in G:|z|=r\}$ is either a circle or is of the form $\{re^{i\theta}: |\theta|< \pi\}$. Then $d\rho_G(re^{i\theta})$ is an increasing function of $|\theta|$ for $0\le |\theta|<\pi$.
\end{lemma}
It follows easily from Lemma~\ref{Weit} that for such a domain $G$ the hyperbolic geodesic joining any two points on the positive real axis is the line segment joining those two points.

In order to prove Example~\ref{extreme-ex} we need two further results. The first is due to Sodin \cite[Lemma~4]{Sodin}, who sharpened an earlier result of this type due to Benedicks~\cite{Ben}. In this result, cap\,$(.)$ denotes logarithmic capacity.
\begin{lemma}\label{sodin}
Let $Q_z(h)$ be the open square in the complex plane with center $z$ and sidelength $h>0$, and let $Q(h)=Q_0(h)$. Let $E\subset Q(r) \cap \R$ be a closed set such that, for some $\delta\in (0,1)$ and $h\in (0,r)$,
\[
{\rm cap}\left( \frac{1}{2h}(E\cap Q_x(h))\right)\ge \delta, \qfor |x| < r-h.
\]
Then
\[
\omega(0) \le \frac{Ch}{r}\log(1/\delta),
\]
where $\omega$ denotes the harmonic measure in $Q(r)\setminus E$ of $\partial Q(r)$, and $C>0$ is an absolute constant.
\end{lemma}
Finally, we need a special case of a beautiful theorem of  Kjellberg~\cite[Theorem~6.7]{wH89}, which generalised and sharpened earlier results of Wiman~\cite{aW05} and Heins~\cite{mH48}.

\begin{lemma}\label{Wiman}
Let~$u$ be a non-constant subharmonic function in $\C$ of order~1/2. Then either
\[
\limsup_{r\to\infty}A(r,u)=\infty,
\]
or
\[
\lim_{r\to\infty}B(r,u)/r^{1/2}=\alpha\;\text{ as } r \to\infty, \quad\text{where } 0<\alpha<\infty. 
\]
\end{lemma}

\section{Proof of Theorem~\ref{basic-props}}\label{lower-estimate-proof}
\setcounter{equation}{0}
The proof of Theorem~\ref{basic-props} is reasonably straightforward. The first statement in part~(a) follows easily from the fact that any subharmonic function $u\in\K$ can be represented as a potential of the form
\[
u(z)=u(0)+\int_0^{\infty} \log|1+z/t|\,d\mu(t),
\]
where $\mu(t)$ denotes the Riesz measure of~$u$ in the disc $\{z:|z| \le t\}$, which is clearly entirely supported in the set~$E$; see \cite[Theorem~2.1]{mH48}, for example.

To prove the second statement in part~(a) we use the fact that the function $U(z)=u(z^2)$ is positive harmonic in the right half-plane with continuous boundary values, so it has the Poisson integral representation
\begin{equation}\label{Poisson-half}
U(z)=cx+\frac{x}{\pi}\int_{-\infty}^{\infty}\frac{U(it)}{|z-it|^2}\,dt,\qfor x=\Re (z)>0,
\end{equation}
where $c\ge 0$; see, for example, \cite[Theorem 7.26]{AB}. It follows that $U(x)/x$ is decreasing for $x\in (0,\infty)$ and hence that $u(r)/r^{1/2}$ is decreasing for $r>0$.

To prove part~(b) we use Lemma~\ref{Beur}. Since $B(r,u)=u(r)$ for $r>0$, by Theorem~\ref{basic-props}, part~(a), this lemma gives
\[
u(r)> \frac 12 \exp \left(\frac 12 \int_{E^*\cap(1,r)} \frac 1t\,dt\right)u(1),\qfor r>1,
\]
and hence
\[
\frac{\log u(r)}{\log r}> \frac{1}{2\log r}\int_{E^*\cap(1,r)} \frac 1t\,dt+\frac{\log u(1)-\log 2}{\log r},\qfor r>1,
\]
from which both statements in part~(b) follow immediately.

{\it Remark}\;\; The above proof can be adapted easily to show that if the set~$E$ is a sufficiently `thick' subset of the negative real axis, then
\begin{equation}\label{mean-type}
u(r)\ge cr^{1/2}, \qfor r>1,
\end{equation}
where~$c$ is a positive constant. For example, suppose that the sequences $(a_n)$ and $(b_n)$ satisfy $a_0=1$ and the recurrence relations
\[
b_n=2a_n\quad \text{and} \quad a_{n+1}=b_n+b^p_n,\qfor n \ge 0,
\]
where $0<p<1$, and $u\in\K$ corresponds to the set $E=\bigcup_{n=0}^{\infty} [-b_n,-a_n]$. Then Lemma~\ref{Beur} gives, for $n\ge 1$,
\begin{align*}
u(b_n)&> \frac 12 \exp\left(\frac 12 \int_{E^*\cap(b_0,b_n)} \frac 1t\,dt\right)u(b_0)\\
&=\frac12 \prod_{j=1}^n\left(\frac{b_j}{a_j}\right)^{1/2} u(b_0)\\
&=\frac12 \prod_{j=1}^n\left(\frac{b_j}{b_{j-1}+b^p_{j-1}}\right)^{1/2} u(b_0)\\
&=\frac12 \left(\frac{b_n}{b_0}\right)^{1/2}\prod_{j=1}^n\left(\frac{1}{1+b^{p-1}_{j-1}}\right)^{1/2} u(b_0),
\end{align*}
from which \eqref{mean-type} follows, since $b_n \ge 2^n$ for $n\ge 0$ and $u(r)/r^{1/2}$ is decreasing for $r>0$.

\section{Proof of Theorem~\ref{exactorder}}\label{upper-estimate-proof}
\setcounter{equation}{0}

In this section we give the proof of Theorem~\ref{exactorder}. First, however, we give a basic technical lemma that will be needed in the proof. Here we use the notation $\log^+ x=\max\{\log x,0\}$ for $x>0$.

\begin{lemma}\label{tech}
Let $x_n>0$ for $n=1,2,\ldots$, and suppose that
\[
\frac1n \sum_{j=1}^n x_j \to \infty \;\text{ as } n\to \infty.
\]
Then
\[
\sum_{j=1}^n \log^+ x_j = o\left(\sum_{j=1}^n x_j\right) \;\text{ as } n\to \infty.
\]
\end{lemma}
\begin{proof}
For $x\ge X\ge e$, we have $(\log x)/x \le (\log X)/X$. Therefore, for $\eps>0$ and $X\ge e$, we have
\begin{align*}
\sum_{j=1}^n \log^+ x_j &= \sum_{x_j\le X} \log^+ x_j+ \sum_{x_j>X} \log x_j\\
&\le n\log X+ \frac{\log X}{X}\sum_{j=1}^n x_j.
\end{align*}
Hence, by further taking~$X$ so large that $(\log X)/X \le \frac12 \eps$ and then $n$ so large that $n\log X \le \frac12 \eps \sum_{j=1}^n x_j$, we obtain
\[
\sum_{j=1}^n \log^+ x_j\le \eps \sum_{j=1}^n x_j,
\]
as required.
\end{proof}
In order to prove Theorem~\ref{exactorder}, given the results of Theorem~\ref{basic-props}, part~(b), we need only show that, under the given hypotheses,
\begin{equation}\label{upper-estimates}
\rho(u)\le \tfrac12\overline{\Lambda}(E^*)\quad \text{and} \quad \lambda(u)\le \tfrac12 \underline{\Lambda}(E^*).
\end{equation}
By Theorem~\ref{basic-props}, part~(a), we have
\[
\rho(u)=\limsup_{r\to\infty}\frac{\log u(r)}{\log r}\quad \text{and} \quad \lambda(u)=\liminf_{r\to\infty}\frac{\log u(r)}{\log r}.
\]
Therefore, by Lemma~\ref{BearSmith} and the definition of the Harnack metric in~\eqref{Har-metric}, it is sufficient to show that
\begin{equation}\label{order-hyp-estimates}
\limsup_{r\to\infty}\frac{\rho_D(1,r)}{\log r} \le \tfrac12\overline{\Lambda}(E^*) \quad \text{and} \quad \liminf_{r\to\infty}\frac{\rho_D(1,r)}{\log r} \le \tfrac12\underline{\Lambda}(E^*),
\end{equation}
where $D=\C\setminus E^*$ as usual.

In view of the remark after Lemma~\ref{Weit}, we have
\begin{equation}\label{segment}
\rho_D(1,r)= \int_1^r d\rho_D(t)\,dt.
\end{equation}
Therefore, to prove \eqref{order-hyp-estimates} we need to obtain good upper estimates for the hyperbolic density $d\rho_D(t)$, for $t>1$.

First, we have the basic hyperbolic density estimate
\begin{equation}\label{hyp-est1}
d\rho_D(t) \le d\rho_{\C\setminus (-\infty,0]}(t) = \frac{1}{2t}, \qfor t>0,
\end{equation}
which follows from the standard monotonicity property of the hyperbolic metric, using the fact that $D\subset \C\setminus (-\infty,0]$, together with an evaluation of $d\rho_{\C\setminus (-\infty,0]}(t)$, for $t>0$, by conformal mapping from the right half-plane to the cut plane.

Second, we have the Beardon--Pommerenke estimate in Lemma~\ref{BeardPomm}. To apply this it is convenient to assume, as we may by the monotonicity property of the hyperbolic metric, that $a_0=0$ and $b_0\ge 1$; that is, the first interval of $E$ contains $[-1,0]$. With this assumption on~$E$ the closest point of $\partial D$ to any positive number~$t$ is 0, so we can apply Lemma~\ref{BeardPomm} to obtain the estimate
\begin{equation}\label{hyp-est2}
d\rho_D(t) \le \frac{\pi/2}{t\beta_D(t)}, \qfor t>0,
\end{equation}
where
\[
\beta_D(t)=\inf\{\left|\log (t/|b|)\right|:b\in \partial D\}, \qfor t>0.
\]
The Beardon-Pommerenke estimate is more effective than \eqref{hyp-est1} when~$t$ lies well inside an interval of the form $[b_n,a_{n+1}]$. Indeed, putting $s_n=\sqrt{b_na_{n+1}}$, $n\ge 0$, we deduce from \eqref{hyp-est2} that, for $n\ge 0$,
\begin{equation}\label{hyp-est3}
\beta_D(t) =
\begin{cases}
\log t/b_n, & b_n <t\le s_n,\\
\log a_{n+1}/t, & s_n\le t<a_{n+1}.
\end{cases}
\end{equation}
To take advantage of this better estimate, we shall apply \eqref{hyp-est1} for values of~$t$ lying in intervals of the form
\[
[a'_n,b'_n], \quad \text{where } a'_n=\tfrac12 a_n \; \text{and} \; b'_n=2b_n, \qfor n\ge 0,
\]
and the estimate \eqref{hyp-est2} in the complementary intervals.

Together with \eqref{segment}, this gives
\begin{equation}\label{up-bound}
\rho_D(1,r) \leq \frac12\int_{E'\cap (1,r)}\frac{dt}{t} + \frac{\pi}{2}\int_{(1,r) \setminus E'} \frac{dt}{t\beta_D(t)},
\end{equation}
where
\[
E' = \bigcup_{n=0}^{\infty}[a'_n,b'_n].
\]
Observe that the complementary intervals of $E'$ are of the form $(b'_n,a'_{n+1})$ in the cases where $b'_n<a'_{n+1}$. Also, for such~$n$, we have $s_n\in (b'_n,a'_{n+1})$.

We now consider each of the two integrals in \eqref{up-bound} separately.

{\bf Claim 1}
\[
\frac12\int_{E'\cap (1,r)}\frac{dt}{t} \leq \tfrac12\overline{\Lambda}(E^*)\log r \,(1 + o(1))\; \text{ as } r \to \infty.
\]
\begin{proof} We note that, for $a'_n \leq r \leq a'_{n+1}$, where $n \geq 0$, we have
\[
\frac12\int_{E'\cap (1,r)}\frac{dt}{t} \leq \frac12\int_{E^*\cap (1,r)}\frac{dt}{t} + n \log 2.
\]
Claim~1 now follows from the fact that
\[
\overline{\Lambda}(E^*)=\limsup_{r \to\infty}\frac{1}{\log r}\int_{E^*\cap(1,r)}\frac{dt}{t},
\]
together with the fact that $n = o(\log r)$ as $r \to \infty$, since $r \geq a_n/2$ and $a_n^{1/n} \to \infty$ as $n \to \infty$, by \eqref{an-cond}.
\end{proof}

Obtaining an upper bound for the second integral requires more work.

{\bf Claim 2}
\[
\frac{\pi}{2}\int_{(1,r) \setminus E'} \frac{dt}{t\beta_D(t)}= o(\log r)\; \text{ as } r \to \infty.
\]
\begin{proof} We consider the case that $a'_n \leq r \leq a'_{n+1}$, where $n \geq 1$. It follows from~\eqref{hyp-est2} and~\eqref{hyp-est3} that
\begin{align*}
\frac{\pi}{2}\int_{(1,a'_n) \setminus E'} \frac{dt}{t\beta_D(t)}
& =\frac{\pi}{2} \sum_{\substack{j=0\\ b'_j<a'_{j+1}}}^{n-1}\left( \int_{b'_j}^{s_j} \frac{dt}{t \log t/b_j} + \int_{s_j}^{a'_{j+1}} \frac{dt}{t \log a_{j+1}/t}  \right) \\
& =  \frac{\pi}{2} \sum_{\substack{j=0\\ b'_j<a'_{j+1}}}^{n-1}\left( \log \frac{\log s_j/b_j}{\log 2} + \log \frac{\log a_{j+1}/s_j}{\log 2} \right)\\
& =  \pi \sum_{\substack{j=0\\ b'_j<a'_{j+1}}}^{n-1} \log \frac{\log a_{j+1}/s_j}{\log 2}\\
& <  \pi \sum_{j=0}^{n-1} \log^+ \log a_{j+1}/a_j,
\end{align*}
since $s_j/b_j=a_{j+1}/s_j=a^{1/2}_{j+1}/b^{1/2}_j\le a^{1/2}_{j+1}/a^{1/2}_j$.

In view of condition \eqref{an-cond}, we can now apply Lemma~\ref{tech} with $x_j = \log a_{j+1}/a_j$ to give
\[
\frac{\pi}{2}\int_{(1,a'_n) \setminus E'} \frac{dt}{t\beta_D(t)} = o\left( \sum_{j=1}^{n-1}\log a_{j+1}/a_j \right) = o(\log a_n) = o(\log r)\; \text{ as } r \to \infty.
\]
This proves Claim 2 in the case that $a'_n \leq r \leq b'_n$.

It remains to show that, if $b'_n \leq r \leq a'_{n+1}$, then
\begin{equation}\label{end}
\frac{\pi}{2}\int_{b'_n}^r \frac{dt}{t \beta_D(t)} = o(\log r)\; \text{ as } r \to \infty.
\end{equation}

We split this into two cases. First, if $b'_n \leq r \leq s_n$, then it follows from~\eqref{hyp-est3} that
\begin{equation}\label{case1}
\frac{\pi}{2}\int_{b'_n}^r \frac{dt}{t \beta_D(t)} = \frac{\pi}{2}\int_{b'_n}^r \frac{dt}{t \log t/b_n} = \frac{\pi}{2} \log \frac{\log r/b_n}{\log 2}= o(\log r)\; \text{ as } r \to \infty.
\end{equation}
Second, if $s_n < r \leq a'_{n+1}$, then it follows from~\eqref{hyp-est3} together with~\eqref{case1} that
\begin{align*}
\frac{\pi}{2}\int_{b'_n}^r \frac{dt}{t \beta_D(t)}
& \leq  \frac{\pi}{2}\int_{b'_n}^{s_n} \frac{dt}{t \log t/b_n} +  \frac{\pi}{2}\int_{s_n}^{a'_{n+1}} \frac{dt}{t \log a_{n+1}/t} \\
& = \frac{\pi}{2} \log \frac{\log s_n/b_n}{\log 2} + \frac{\pi}{2} \log \frac{\log a_{n+1}/s_n}{\log 2}\\
& = \pi \log \frac{\log s_n/b_n}{\log 2} < \pi \log \frac{\log r}{\log 2} = o(\log r)\; \text{ as } r \to \infty.
\end{align*}
Together with~\eqref{case1}, this shows that~\eqref{end} is true. This completes the proof of Claim 2.
\end{proof}

It follows from Claim 1 and Claim 2 together with~\eqref{up-bound} that
\[
\rho_D(1,r) \leq \tfrac12\overline{\Lambda}(E^*)\log r (1 + o(1))\; \text{ as } r \to \infty.
\]
The first estimate in~\eqref{order-hyp-estimates} now follows.

A similar but somewhat simpler argument can be used to prove the second statement in~\eqref{order-hyp-estimates}. First recall that
\[
\underline{\Lambda}(E^*)=\liminf_{r \to\infty}\frac{1}{\log r}\int_{E^*\cap(1,r)}\frac{dt}{t}.
\]
It is easy to check that, for $n\ge 0$,
\[
\min_{b_n\le r\le b_{n+1}}\frac{1}{\log r}\int_{E^*\cap(1,r)}\frac{dt}{t}
\]
occurs at $r=a_{n+1},$ so there must be a subsequence $a_{n_k}, k=1,2,\ldots,$ such that
\[
\underline{\Lambda}(E^*)=\lim_{k \to\infty}\frac{1}{\log a_{n_k}}\int_{E^*\cap (1,a_{n_k})}\frac{dt}{t}.
\]
Hence, by similar reasoning to that used to prove Claim~1 and by Claim~2 (in the special case when $r=a_{n_k}, k=1,2,\ldots$), we deduce that
\begin{align*}
\rho(1,a_{n_k}) &\le \frac12\int_{E^*\cap (1,a_{n_k})}\frac{dt}{t}  + n_k\log 2+ \frac{\pi}{2}\int_{(1,a_{n_k}) \setminus E'} \frac{dt}{t\beta_D(t)}\\
&\le \tfrac12\underline{\Lambda}(E^*)\log a_{n_k}(1+o(1)) + o\left( \log a_{n_k} \right)\\
&\le \tfrac12\underline{\Lambda}(E^*)\log a_{n_k} (1 +o(1)) \;\text{ as } k \to \infty.
\end{align*}
The second estimate in~\eqref{order-hyp-estimates} now follows.

This completes the proof of Theorem~\ref{exactorder}.

\section{Proof of Example~\ref{extreme-ex}}\label{example-proof}
\setcounter{equation}{0}
Our example is a function $u\in\K$ where
\[
E=\bigcup_{n\ge 1}[-b_n,-a_n], \quad\text{with } a_n=n, b_n=n+1/n,\; n=1,2, \ldots,
\]
normalised so that $u(1)=1$. It is straightforward to check that $\overline{\Lambda}(E^*)=0$ in this case.

We shall show that, for this function,
\begin{equation}\label{o(1)}
u(-r)=o(1)\;\text{ as } r\to \infty.
\end{equation}
It follows, by Barry's theorem (see the end of Section~\ref{intro}), that~$u$ must have order~$1/2$ and then, by Lemma~\ref{Wiman}, that~$u$ cannot be of order~1/2, minimal type. Hence $\lim_{r\to\infty}u(r)/r^{1/2}>0$, by Theorem~\ref{basic-props}, part~(a); in particular,~$u$ has order and lower order~$1/2$.

To prove \eqref{o(1)}, we first apply Lemma~\ref{sodin} to the part of $E$ that lies in the square box of the form
\[
Q_{z_r}(\tfrac12 r)=\{z:-2r <\Re z<-r:|\Im z|<\tfrac12 r\},\quad\text{where }  z_r=-\tfrac32 r, r>1.
\]
From now on we assume that $r>4$. Then
\[
\max\{n\in\N: -b_n\in Q_{z_r}(\tfrac12 r)\}\le 2r,
\]
so, for $r<a_n<b_n<2r$,
\[
b_n-a_n = \frac1n \ge \frac{1}{2r}.
\]
We shall apply Lemma~\ref{sodin} with $h=1$. In this case, for $|x-z_r| < \tfrac12r-h$,
\[
E\cap Q_x(h)\text{ contains at least one interval of } E,
\]
so
\[
{\rm cap}\left( \frac{1}{2h}(E\cap Q_x(h))\right)\ge \frac14\cdot \frac12\cdot \frac{1}{2r}=\frac{1}{16r}, \qfor |x| < \tfrac12 r-h,
\]
since cap\,$(I) \ge \frac14 |I|$ for any interval $I$ on the real line; see~\cite[Corollary~9.10]{Pomm}, for example. Therefore, by Lemma~\ref{sodin}, with $h=1$ and $\delta=1/(16r)$, we have
\begin{equation}\label{omega}
\omega(z_r) \le \frac{C}{r}\log(16r), \qfor r>4,
\end{equation}
where $\omega$ denotes the harmonic measure in $Q_{z_r}(\tfrac12 r)\setminus E$ of $\partial Q_{z_r}(\frac12r)$ and $C>0$ is an absolute constant.

Now note that, by Theorem~\ref{basic-props}, part~(a),
\[
\max\{u(z): z\in \partial Q_{z_r}(\tfrac12r)\}\le u(3r) \le (3r)^{1/2},\qfor r>0.
\]
It follows, by applying the maximum principle to~$u$ in $Q_{z_r}(\tfrac12 r)\setminus E$ and using \eqref{omega} that, for $r>4$,
\begin{align*}
u(z_r)&\le \max\{u(z): z\in \partial Q_{z_r}(\tfrac12r)\} \omega(z_r)\\
&\le \frac{C(3r)^{1/2}}{r}\log(16r)\\
&=o(1) \;\text{ as } r \to \infty,
\end{align*}
as required.

\section{Proof of Theorem~\ref{min-type}}\label{min-type-result}
\setcounter{equation}{0}
For the proof of Theorem~\ref{min-type}, we need the following result on positive harmonic functions defined in annuli; see \cite[Theorem~3.1]{BRS11} for a variation on Lemma~\ref{Har}.

\begin{lemma}\label{Har}
Let~$u$ be positive and harmonic in $\{z:r_1<|z|<r_2\}$ and suppose that $r_1<s_1\le s_2<r_2$. Then there is a positive constant $K$ depending only on $\mu:=\min\{\log(s_1/r_1), \log(r_2/s_2)\}$ such that
\begin{equation}\label{annulus-est}
u(z')\le Ku(z), \qfor |z'|=|z|\in [s_1,s_2].
\overline{}\end{equation}
\end{lemma}
\begin{proof}
The positive harmonic function $u(e^t)$ is defined in the infinite strip  $S=\{t:\log r_1<\Re (t)<\log r_2\}$. We can apply Harnack's inequality to this function in any disc of radius~$\mu$ whose centre lies in the rectangle
\[
R=\{t:\log s_1\le \Re (t)\le \log s_2,-\pi\le \Im (t)\le \pi\}\subset S,
\]
to deduce that there is a positive constant $K=K(\mu)$ such that
\[
u(e^{t'})\le Ku(e^t),\qfor t,t'\in R,\; \Re (t')=\Re (t)\in [\log s_1, \log s_2],
\]
and this gives \eqref{annulus-est}.
\end{proof}

Theorem~\ref{min-type} states that if the function $u\in \K$, with corresponding closed subset $E$ of the negative real axis, satisfies
\[E^c\supset\bigcup_{n\ge 0}(-d_n,-c_n),\]
where $0\le c_0<d_0<c_1<d_1< \cdots,$ and $\limsup_{n\to\infty} d_{n}/c_n>1$, then
\[
\frac{u(r)}{r^{1/2}}\to 0\;\;\text{as }r\to\infty.
\]
We give two proofs, the first based on Lemma~\ref{Wiman} and the other a direct one using the Poisson integral formula.

\begin{proof}[First proof of Theorem~\ref{min-type}]
By the hypotheses on $(c_n)$ and $(d_n)$, we can assume that there exists $d>1$ such that
\begin{equation}\label{dncn}
d_n/c_n \ge d, \qfor n\ge 0.
\end{equation}
We will apply Lemma~\ref{Har} with
\[
r_1=c_n,\quad s_1=c_n^{3/4}d_n^{1/4},\quad s_2=c_n^{1/4}d_n^{3/4},\quad r_2= d_n.
\]
Then, by~\eqref{dncn},
\[
\frac{s_1}{r_1}=\frac{r_2}{s_2}=\left(\frac{d_n}{c_n}\right)^{1/4}\ge d^{\,1/4},\qfor n\ge 0.
\]
We deduce that there is a constant $K=K(d)>0$ such that
\begin{equation}\label{u(s)-est}
u(r) \le Ku(-r),\qfor c_n^{3/4}d_n^{1/4}\le r \le c_n^{1/4}d_n^{3/4},\;n\ge 0,
\end{equation}
and, in particular, $A(r)=u(-r), r>0,$ is unbounded, by Theorem~\ref{basic-props}, part~(a).

The fact that $u(r)=B(r) = o(r^{1/2})$ as $r\to\infty$ now follows immediately from Lemma~\ref{Wiman} and Theorem~\ref{basic-props}, part~(a), in the case that $\rho(u)=1/2$ and is trivial if $\rho(u)<1/2$.
\end{proof}

\begin{proof}[Second proof of Theorem~\ref{min-type}]
The alternative direct argument uses the estimate
\begin{equation}\label{u(r)-est}
u(r) \ge \frac{r^{1/2}}{\pi}\int_0^{\infty} \frac{u(-s)}{s^{1/2} (s+r)}\,ds \ge \sum_{n= 0}^{\infty} \frac{r^{1/2}}{\pi}\int_{c_{n}}^{d_n}\frac{u(-s)}{s^{1/2} (s+r)}\,ds,
\end{equation}
which follows from the convergence of the integral in \eqref{Poisson-half} after the change of variable $z\mapsto \sqrt z$; note that $u\in\K$ is symmetric with respect to the real axis.

Now let $\alpha_n=c_n^{3/4}d_n^{1/4}$ and $\beta_n=c_n^{1/4}d_n^{3/4}$. We deduce from \eqref{u(r)-est} with $r=1$, together with \eqref{dncn}, \eqref{u(s)-est} and the fact that $u(r)/r^{1/2}$ is decreasing (by Theorem~\ref{basic-props}, part~(a)) that
\begin{align}
\pi u(1)&>\sum_{n\ge 0} \int_{c_n}^{d_{n}} \frac{u(-s)}{2s^{3/2}}\,ds \ge \frac1K\sum_{n\ge 0} \int_{\alpha_n}^{\beta_n} \frac{u(s)}{2s^{3/2}}\,ds\notag\\
& \ge \frac{1}{2K}\sum _{n\ge 0} \frac{u(\beta_n)}{\beta_n^{1/2}}\log \left(\frac{\beta_n}{\alpha_n}\right)\ge \frac{\log d^{\,1/2}}{2K} \sum _{n\ge 0} \frac{u(\beta_n)}{\beta_n^{1/2}}.\notag
\end{align}
Hence $u(\beta_n)/\beta_n^{1/2}\to 0$ as $n\to \infty$, so $u(r)/r^{1/2}\to 0$ as $r\to \infty$, as required.
\end{proof}

{\it Remarks}\; 1. The example in the remark following the proof of Theorem~\ref{basic-props} and also Example~\ref{extreme-ex} show that we cannot hope to significantly weaken the condition $\limsup_{n\to\infty} d_{n}/c_n>1$ in Theorem~\ref{min-type}.

2. Theorem~\ref{min-type} can be used to show that the functions $u_{\alpha,\beta}$ considered by Kjellberg have order strictly less than~1/2. Indeed, for these functions the gaps between the intervals in the set~$E$ satisfy the hypotheses of Theorem~\ref{min-type}, so we certainly have $u_{\alpha,\beta}(r)=o(r^{1/2})$ as $r\to\infty$.

However, in view of the fact that the set~$E$ is invariant under scaling by $z\mapsto \beta z$ and the uniqueness property of functions $u\in \K$, we have $u_{\alpha,\beta}(\beta z)=Cu_{\alpha,\beta}(z)$ for all $z\in \C$, where $C=C(\alpha,\beta)>0$, and this property is incompatible with $u_{\alpha,\beta}$ having order~1/2, minimal type.

\section{Proofs of Theorem~\ref{discretise} and Corollary~\ref{entire-density}}\label{Kjellbergproof}
\setcounter{equation}{0}
The proofs of Theorem~\ref{discretise} and Corollary~\ref{entire-density} follow the structure of the reasoning in \cite[Chapter~4]{bK48} but require significant additional arguments due to the much greater generality of the sets~$E$ considered here.
\begin{proof}[Proof of Theorem~\ref{discretise}]
Once again, we express the subharmonic function $u\in\K$ as a potential of the form
\[
u(z)=u(0)+\int_0^{\infty} \log|1+z/t|\,d\mu(t),
\]
where $\mu(t)$ denotes the Riesz measure of the subharmonic function~$u$ in the disc $\{z:|z| \le t\}$, clearly entirely supported in the set~$E$. It is well known (see \cite[Chapter~3]{HK76}) that
\begin{equation}\label{mu-u}
\mu(r)=rI'(r),\quad\text{where } I(r)= \frac{1}{2\pi}\int_0^{2\pi} u(re^{i\theta})\,d\theta,\qfor r\ge 0.
\end{equation}
As noted earlier, the idea of Kjellberg's method is to discretise the measure $\mu$ and use the resulting discrete measure to construct a subharmonic function~$u_1$ which is close to~$u$ in much of the plane, and for which $u_1(z)=\log |f(z)|$, where~$f$ is an entire function.

Indeed, writing
\[
u_1(z)=u(0)+ \int_0^{\infty} \log|1+z/t|\,d[\mu(t)] = u(0)+\sum_{n=1}^{\infty} \log |1+z/x_n|,
\]
where the sequence $(x_n)$ is positive and increasing, and $-x_n\in E$, we see that $u_1(z)=\log |f(z)|$, where $f$ is the entire function
\[
f(z)=C\prod_{n=1}^{\infty}(1+z/x_n),\quad C=e^{u(0)}.
\]
Now recall from the statement of Theorem~\ref{discretise} that
\[
D_1=\C\setminus \{z: \text{{\rm dist}}(z,E)\le 1\}.
\]
We shall show that there exists a positive constant $R_0=R_0(u)$ such that
\begin{equation}\label{u-u1}
|u(z)-u_1(z)|\le 4\log|z|, \qfor z\in D_1, |z|\ge R_0,
\end{equation}
which gives \eqref{R-est}.

To do this, we write
\begin{align}
u(z)-u_1(z)&=\int_0^{\infty} \log|1+z/t|\,d(\mu(t)-[\mu(t)])\notag\\
&= \left(\int_0^{x_1} + \int_{x_1}^{\infty}\right)\log|1+z/t|\,d(\mu(t)-[\mu(t)])\notag\\
&= I_1(z)+I_2(z),
\end{align}
say, and estimate $I_1(z)$ and $I_2(z)$ in turn, for $z\in D_1$.

Since $\mu(x_1)=1$, we have
\begin{align}\label{I1-est1}
I_1(z)&=\int_0^{x_1} \log|1+z/t|\,d\mu(t)-\log|1+z/x_1|\notag\\
&=\int_0^{x_1} \log|t+z|\,d\mu(t)-\int_0^{x_1} \log t\,d\mu(t)-\log|1+z/x_1|,
\end{align}
provided that the two integrals in the latter expression are convergent. This is the first point at which our proof requires different reasoning from that in \cite{bK48}.

To prove this convergence, we use the fact that $I(r)$, defined in \eqref{mu-u}, is a non-negative increasing convex function of $\log r$, that is, $\phi(s)=I(e^s)=I(r), s\in\R,$ is a non-negative increasing convex function. Hence $rI'(r)\log r=s\phi'(s)\to 0$ as $s\to -\infty$, by a simple argument. Therefore, by \eqref{mu-u}, we deduce that the integral
\begin{align*}
\int_0^{x_1} \log t\,d\mu(t)&=[tI'(t)\log t]^{x_1}_0-\int_0^{x_1}I'(t)\, dt\\
&= x_1I'(x_1)\log x_1-I(x_1)+I(0)
\end{align*}
is convergent.

Since $|z+t|>1$ for $z\in D_1$ and $-t\in E$, we deduce that, for $z\in D_1$,
\begin{equation}\label{I1-est2}
0<\int_0^{x_1} \log|t+z|\,d\mu(t)\le \log(x_1+|z|)\int_0^{x_1}\,d\mu(t) = \log(x_1+|z|).
\end{equation}
Hence, by \eqref{I1-est1} and \eqref{I1-est2}, there exists $R_1>0$ and $C_1>0$ such that
\begin{equation}\label{I1-est3}
\left|I_1(z)\right| \le C_1, \qfor z\in D_1, |z|\ge R_1.
\end{equation}
Next we estimate $|I_2(z)|$ for $z\in D_1$, and for this second integral we need to give considerably more detail than was given in \cite{bK48}. Integrating by parts, we obtain
\begin{align}\label{int-by-parts}
I_2(z)&=\int_{x_1}^{\infty} \log|1+z/t|\,d(\mu(t)-[\mu(t)])\notag\\
&=-\int_{x_1}^{\infty}(\mu(t)-[\mu(t)])\,d\log|1+z/t|,
\end{align}
since $\mu(x_1)=[\mu(x_1)]$, so
\begin{equation}\label{I2-est1}
|I_2(z)|\le\int_{x_1}^{\infty}\left|d\log|1+z/t|\right|.
\end{equation}
Now, for $z$ in the right half-plane and $t>0$, the function $t\mapsto \log|1+z/t|$, $t>0$, is decreasing (as $t$ increases), so
\begin{equation}\label{I2-est2}
\int_{x_1}^{\infty}\left|d\log|1+z/t|\right|=-\int_{x_1}^{\infty}d\log|1+z/t|=\log|1+z/x_1|.
\end{equation}
Now we assume that $z=x+iy$ lies in the left half-plane, and observe that
\begin{equation}\label{inc-dec}
t\mapsto \log|1+z/t| \text{ is }
\begin{cases}
\text{decreasing for } 0< t\le |z|^2/|x|,\\
\text{increasing for } t\ge |z|^2/|x|;
\end{cases}
\end{equation}
the change from decreasing to increasing occurs at the value $t=|z|^2/|x|$ where $1+z/t$ is orthogonal to~$z$. Hence, for~$z$ in the left half-plane with $|z|\ge x_1$,
\begin{align*}
\int_{x_1}^{\infty}\left|d\log|1+z/t|\right|&=-\int_{x_1}^{|z|^2/|x|}d\log|1+z/t|+\int^{\infty}_{|z|^2/|x|}d\log|1+z/t|\\
&=\log|1+z/x_1|-2\log\left|1+z|x|/|z|^2\right|\\
&=\log|1+z/x_1|-2\log|y|/|z|.
\end{align*}
It follows from this estimate, and \eqref{I2-est1} and \eqref{I2-est2}, that if $z\in D_1\setminus \{z=x+iy:x< 0,|y|< 1/2\}$ and $ |z|\ge x_1$, then
\begin{equation}\label{I2-est3}
|I_2(z)|\le \left|\log|1+z/x_1|\right|+2\log |z|+2\log 2.
\end{equation}
Finally, suppose that $z\in D_1\cap \{z=x+iy:x < -b_0, |y|< 1/2\}$. Then $z$ lies in a set of the form
\[
\Omega_n=\{z:|z+b_n|>1,|z+a_{n+1}|>1, |y|<1/2\},\;\text{ where }
n\ge 0,
\]
and \eqref{int-by-parts} can be written as
\[
I_2(z)=-\int_{x_1}^{b_n}(\mu(t)-[\mu(t)])\,d\log|1+z/t|-\int_{a_{n+1}}^{\infty}(\mu(t)-[\mu(t)])\,d\log|1+z/t|,
\]
since the measure $\mu$ is supported in $E$.

Also, for $n$ sufficiently large and $z\in \Omega_n$, we have $b_n<|z|^2/|x|<a_{n+1}$, so in this case,
\begin{align}\label{I2-est4}
|I_2(z)|&\le\int_{x_1}^{b_n}\left|d\log|1+z/t|\right|+\int_{a_{n+1}}^{\infty}\left|d\log|1+z/t|\right|\notag\\
&=\log|1+z/x_1|-\log\left|1+z/b_n\right|-\log\left|1+z/a_{n+1}\right|,
\end{align}
in view of \eqref{inc-dec}.

Next we observe that, if $z\in \Omega_n$, for $n\ge 0$, then
\begin{equation}\label{z+bn}
|1+z/b_n|\ge 1/|z|\quad \text{and} \quad |1+z/a_{n+1}| \ge 1/(2|z|).
\end{equation}
The first estimate is immediate since $|z|>|b_n|$ and $|z+b_n|>1$ for $z\in \Omega_n$. The second estimate holds because $|z||a_{n+1}+z| \ge |x|(a_{n+1}+|x|)>a_{n+1}/2$, for $z=x+iy \in \Omega_n$, by an elementary calculation.

Therefore, by \eqref{I2-est4} and \eqref{z+bn}, we deduce that \eqref{I2-est3} holds for all values of $z\in D_1$ as long as $|z|$ is large enough.

Combining \eqref{I1-est3} and \eqref{I2-est3}, we deduce that there exist $R_2>0$ and $C_2>0$ such that
\begin{equation}\label{u-u1-est}
|u(z)-u_1(z)|\le 3\log |z|+C_2,\qfor z\in D_1, |z|\ge R_2,
\end{equation}
which gives \eqref{u-u1}. This proves the first part of Theorem~\ref{discretise}.

To prove the second part, we first claim that if \eqref{d-cond} holds, that is, $b_n/a_n \ge d>1$ for $n\ge 0$, then there exists $C_3>0$ such that
\begin{equation}\label{B-est}
u(z)\le C_3,\;\; \text{whenever }\text{{\rm dist}}(z,E)\le 1.
\end{equation}
We use Lemma~\ref{Beur} again. First, let~$n$ be so large that $a_n(1-1/\sqrt d)>1$, and consider a general point $-s\in [-b_n,-a_n]$. We have $s(1-1/\sqrt d)>1$ and $u=0$ on at least one of the intervals $[-s\sqrt d,-s]$ or $[-s, -s/\sqrt d\,]$. Thus, if we apply Lemma~\ref{Beur} with the origin moved to $-s$ and the radii $r_1=1$, $r_2=s-s/\sqrt d>1$, then we obtain
\[
\max_{|z+s|\le s-s/\sqrt d}u(z)\ge \frac12\left(\frac{r_2}{r_1}\right)^{1/2}\max_{|z+s|\le 1}u(z) = \frac12s^{1/2}\left(1-1/\sqrt d\right)^{1/2}\max_{|z+s|\le 1}u(z).
\]
Also, by Theorem~\ref{basic-props},  part~(a),
\[
u(z) \le |z|^{1/2}u(1)\le \left(s\sqrt d\right)^{1/2}u(1),\qfor |z| \le  s\sqrt d,
\]
so the claim \eqref{B-est} follows, since $\{z:|z+s|\le s-s/\sqrt d\,\}\subset \{z:|z|\le s\sqrt d\,\}$.

Now let $v(z)=u_1(z)-(3\log|z|+C_2)$. By \eqref{u-u1-est} and \eqref{B-est}, we deduce that
\[
v(z)\le u(z) \le C_3,
\]
when $z$ lies on the boundary of any complementary component of $D_1$.  Since~$v$ is subharmonic in $\C\setminus \{0\}$, we deduce by the maximum principle that $v(z)\le C_3$ in each complementary component of $D_1$, except possibly one that contains 0, and hence there exists $R=R(u)>0$ such that
\begin{equation}\label{fu2}
\log|f(z)|=u_1(z) \le u(z)+4\log |z|, \qfor |z|\ge R.
\end{equation}
This completes the proof of Theorem~\ref{discretise}.
\end{proof}

\begin{proof}[Proof of Corollary~\ref{entire-density}]
We shall assume that the required order $\rho\in (0,1/2)$ and consider the function $u\in\K$ with $E=\cup_{n=0}^{\infty}[-b_n,-a_n]$, where
\[
b_n=\exp(n^2/(4\rho)) \quad \text{and}\quad a_n=b_n e^{-n}, \qfor n\ge 0,
\]
and also $u(0)=1$. The proof for the case $\rho=0$ us similar, but with $b_n=\exp(n^3)$.

Then $a_n^{1/n} \to \infty$ as $n\to \infty$ and it is easy to check that the logarithmic density of~$E$ is $2\rho$, so the order and lower order of~$u$ is~$\rho$, by Theorem~\ref{exactorder}. Also,
\[
b_n/a_n\to \infty \;\text{ as } n\to \infty,
\]
so we can apply Theorem~\ref{discretise} to obtain an entire function of the form
\[
f(z)=\prod_{n=1}^{\infty}(1+z/x_n),
\]
where the sequence $(x_n)$ is positive and increasing, and $-x_n\in E$, such that~$f$ has order and lower order $\rho$ and
\[
\log|f(z)|\le u(z)+4\log |z|, \qfor |z|\ge R,
\]
where $R=R(u)>0$. If we then redefine~$f$ to be
\[
f(z)=\prod_{n=6}^{\infty}(1+z/x_n),
\]
then~$f$ again has order~$\rho$ and
\[
\log|f(z)|\le u(z), \qfor |z|\ge R,
\]
for some possibly larger $R$.

Since $u=0$ on the set $E$, we deduce that $\{r: A(r,\log|f|)>0)\}$ is a subset of $E^c \cup \{x: -R<x\le 0\}$, and so has lower logarithmic density at most $1-2\rho$, and hence exactly $1-2\rho$, by Barry's theorem \eqref{Barry-order}, applied with $\alpha=1/2$.

The set $\{r: A(r,\log|f|)>0)\}$ has upper logarithmic density at most $1-2\rho$ also. Hence it has upper logarithmic density exactly  $1-2\rho$ by another theorem of Barry \cite{Ba64}, applied with $\alpha=1/2$ again, which states that if~$u$ is a non-constant subharmonic function of lower order $\lambda\in [0,1)$ and $\lambda<\alpha<1$, then
\begin{equation}\label{Barry-lowerorder}
\overline{\Lambda}(\{r: A(r,u)>\cos (\pi\alpha) B(r,u)\}) \ge 1-\lambda/\alpha.
\end{equation}
This completes the proof of Corollary~\ref{entire-density}.
\end{proof}

\end{document}